\immediate\write18{bibtex references}

\documentclass[11pt,twoside,a4paper]{article}
\usepackage{fancyheadings}
\usepackage{datetime}\usdate
\usepackage[latin1]{inputenc}
\usepackage[english]{babel}
\usepackage{graphicx}
\usepackage{dcolumn}
\usepackage{fancyhdr}
\usepackage{amsmath}
\usepackage{verbatim}
\usepackage{amsfonts}
\usepackage{amsthm}
\theoremstyle{plain}
\usepackage{bm}
\usepackage{amssymb} 
\usepackage{color}
\usepackage[hang,small,bf]{caption}    
\usepackage{mciteplus}

\title{ON THE CONNECTION BETWEEN THE CONJUGATE GRADIENT METHOD AND
  QUASI-NEWTON METHODS ON QUADRATIC PROBLEMS}
\author{Anders FORSGREN\thanks{\footTO} \\ Tove
  ODLAND\addtocounter{footnote}{-1}\footnotemark}

\def\footTO{Optimization and Systems Theory, Department of
            Mathematics, KTH Royal Institute of Technology, SE-100 44
            Stockholm, Sweden ({\tt andersf@kth.se, odland@kth.se}).  Research
            supported by the Swedish Research Council (VR).}

\date  {The final publication is available at Springer via http://dx.doi.org/10.1007/s10589-014-9677-5}
\newcommand{\myparallel}{ \ / \! / \ }

\begin{document}
\maketitle\thispagestyle{empty}
\begin{abstract}
It is well known that the conjugate gradient 
method and a quasi-Newton method, using any well-defined 
update matrix from the one-parameter Broyden family of updates, 
produce identical iterates on a quadratic problem with 
positive-definite Hessian. This equivalence does not hold 
for any quasi-Newton method. We define precisely the
conditions on the update matrix in the quasi-Newton method that give rise to
this behavior. We show that the crucial facts are, that the range of each update
matrix lies in the last two dimensions of the Krylov subspaces defined by the 
conjugate gradient method and that the quasi-Newton condition is
satisfied. In the framework based on a sufficient condition
to obtain mutually conjugate search directions, we show that the
one-parameter Broyden family is complete.

A one-to-one correspondence between the Broyden parameter and the
non-zero scaling of the search
direction obtained from the corresponding quasi-Newton method
compared to the one obtained in the
conjugate gradient method is derived.
In addition, we show that the update
matrices from the one-parameter Broyden family are almost always
well-defined on a quadratic problem with positive-definite
Hessian. The only exception is when the symmetric rank-one update is
used and the unit steplength is taken in the same iteration. In this
case it is the Broyden parameter that becomes undefined.
\linebreak

\end{abstract}

\newtheorem{theorem}{Theorem}[section]
\newtheorem{lemma}[theorem]{Lemma}
\newtheorem{assumption}[theorem]{Assumption}
\newtheorem{corollary}[theorem]{Corollary}
\newtheorem{definition}[theorem]{Definition}
\newtheorem{proposition}[theorem]{Proposition}

\pagestyle{fancy}

\fancyhead{}
\fancyhead[RO, LE] {\thepage} 
\fancyhead[CO]{\leftmark}
\fancyhead[CE]{On the connection between CG and QN}
\fancyfoot{}

\section{Introduction}
\label{sec:introduction}
In this paper we examine some well-known methods used for solving unconstrained optimization
problems and specifically their behavior on quadratic problems. A
motivation why these problems are of interest is that the task of solving a linear system of equations
$Ax=b$, with the assumption $A=A^T \succ 0$, may equivalently be considered as the
one of solving an unconstrained quadratic programming problem, 
\begin{equation}\label{qp}
\min_{x \in \mathbb{R}^n}q(x)=\min_{x \in \mathbb{R}^n}\frac{1}{2}x^THx+c^Tx,\tag{QP}
\end{equation}
where one lets $H=A$ and $c=\nolinebreak-b$ to obtain the usual notation. 

Given an initial guess $x_0$, the general
idea of most methods for solving \eqref{qp} is to, in each iteration $k$, generate a search direction $p_k$ and
then take a steplength $\alpha_k$ along that direction to approach the
optimal solution. For $k \geq 0$, the next iterate is hence obtained as
\begin{equation}\label{nextiterate}
x_{k+1}=x_k+\alpha_kp_k.
\end{equation}
The main difference between methods is the way the search 
direction $p_k$ is generated. For high-dimensional problems it is
preferred that only function and gradient values are used in 
calculations. The gradient of the objective function
$q(x)$ is given by $g(x)=Hx+c$ and its value at $x_k$ is denoted by
$g_k$. 

The research presented in this paper stems from the desire to better
understand the well-known connection between the \emph{conjugate 
gradient method}, henceforth CG, and \emph{quasi-Newton methods},
henceforth QN. We are interested in determining precise conditions in
association with the generation of $p_k$ in QN such that, using
exact linesearch, CG and QN will generate the same sequence 
of iterates as they approach the optimal solution of \eqref{qp}. 

CG and QN are introduced briefly in Section \ref{sec:background} where we also state some
background results on the connection between
these two methods. In Section \ref{sec:results} we present our results and some concluding remarks
are made in Section \ref{sec:conclusion}.

\section{Background}
\label{sec:background}
On \eqref{qp}, naturally, we consider the use of \emph{exact
  linesearch}. Then in iteration $k$, the optimal steplength is given by
\begin{equation}\label{steplength}
\alpha_k=-\frac{p_k^Tg_k}{p_k^TH p_k}.
\end{equation}
Since $H=H^T \succ 0$, it can be shown that the descent property, 
$q(x_{k+1}) < q(x_{k})$, holds, as long as
\begin{equation}\label{nonorth}
p_k^Tg_k \neq 0.
\end{equation}
Note that the usual requirement $p_k^Tg_k<0$, i.e. that is $p_k$ is a \emph{descent
  direction} with respect to $q(x)$ at $x_k$ with steplength $\alpha_k >0$, is included in
the more general statement of \eqref{nonorth}. 

Given two parallel vectors and an initial point $x_k$, performing exact
linesearch from the initial point with respect to a given objective function along these two
vectors will yield the same iterate $x_{k+1}$. Hence, two methods will find the
same sequence of iterates if and only if the search directions generated
by the two methods are parallel. Therefore, in the remainder of this
paper our focus will be on parallel search directions rather than
identical iterates. We will denote parallel vectors by $p \myparallel p'$.

On \eqref{qp}, one is generally interested in methods for which
the optimal solution is found in at most $n$ iterations. It can be
shown that a sufficient property for this behavior is
that the method generates search directions which are
mutually conjugate with respect to $H$, i.e. $p_i^THp_j=0, 
\forall i \neq j$, see, e.g., \cite[Chapter 5]{nocedalwright}. 

\subsection{Conjugate gradient method} \label{sec:cg}
A generic way to generate conjugate vectors is by means of the
conjugate Gram-Schmidt process. Given a set of linearly 
independent vectors $\{a_0, \dots, a_{n-1}\}$, 
a set of vectors $\{p_0, \dots,p_{n-1}\}$ mutually conjugate with 
respect to $H$ can be constructed by letting $p_0=a_0$
and for $k>0$,
\begin{equation}\label{cdm}
p_k=a_k+\sum_{j=0}^{k-1}\beta_{kj}p_j.
\end{equation}
The values of $\{\beta_{kj}\}_{j=0}^{k-1}$ are uniquely determined in order to make $p_k$
conjugate to $\{p_0, \dots, p_{k-1}\}$ with respect to $H$. \emph{Conjugate
direction methods} is the common name for all methods which
are based on generating search directions in this manner.

With the choice $a_k=-g_k$ in \eqref{cdm} one obtains the conjugate gradient
method, CG, of Hestenes and Stiefel \cite{HestenesStiefel}. In effect, in CG let $p_0=-g_0$, and for $k>0$ the only $\beta$-value in \eqref {cdm} that will be non-zero is
\begin{equation}\label{beta}
\beta_{k,k-1}=\frac{p_{k-1}^THg_k}{p_{k-1}^TH p_{k-1}},
\end{equation}
where one may drop the first sub-index. It can be shown that
$\beta_{k-1}=g_k^Tg_k/g_{k-1}^Tg_{k-1}$, so that \eqref{cdm} may be
written as 
\begin{equation}\label{cgdir2}
p_k=-g_k+\beta_{k-1}p_{k-1}=-\sum_{i=0}^k\frac{g_k^Tg_k}{g_i^Tg_i}g_i.
\end{equation}
From the use of exact linesearch it holds that $g_k^Tp_j=0$, for all
$j\leq k-1$, which implies that $g_k^Tg_j=0$, for all
$ j\leq k-1$, i.e. the produced gradients are orthogonal and therefore linearly
independent, as required. As for any conjugate
directions method using exact linesearch, the search
directions of CG are descent directions.   See,
e.g., \cite{cgwopain} for an intuitive introduction to the
conjugate gradient method with derivation of the above relations.

In CG, one only needs to 
store the most recent previous search direction, $p_{k-1}$. This is a reduction in
the amount of storage required compared to a general conjugate direction
method where potentially all previous search directions are needed to
compute $p_k$.

Although equations \eqref{nextiterate}, \eqref{steplength},
\eqref{beta} and \eqref{cgdir2} give a complete description of an
iteration of CG, the
power and richness of the method is somewhat clouded in notation. An
intuitive way to picture what happens in an iteration of CG is to
describe it as a Krylov subspace method. 

\begin{definition}\label{krylov}
Given a matrix $A$ and a vector $b$ the \emph{Krylov subspace}
generated by $A$ and $b$ is
given by $\mathcal{K}_{k}(b,A) =span\{ b, Ab, \ldots, A^{k-1}b \}$.
\end{definition}
Krylov subspaces are linear subspaces, which are expanding,
i.e. $\mathcal{K}_{1}(b, A) \subseteq \mathcal{K}_{2}(b, A) \subseteq
\mathcal{K}_{3}(b, A) \subseteq
\dots$, and $\dim(\mathcal{K}_{k}(b,A))=k$. Given $x \in
\mathcal{K}_{k}(b, A)$, then $Ax \in
\mathcal{K}_{k+1}(b, A)$, see, e.g., \cite{gutknecht} for an introduction to
Krylov space methods. CG is a Krylov subspace method and 
iteration $k$ may be formulated as the following constrained
optimization problem
\begin{equation}\label{qpcg}
\min \ q(x), \ \text{s.t.} \ x \in x_0+\mathcal{K}_{k+1}(p_0,H),\tag{CG$_k$}
\end{equation}
see, e.g., \cite[Chapter 5]{nocedalwright}. The optimal solution of (CG$_k$) is $x_{k+1}$ and the corresponding multiplier is
given by \nolinebreak{$g_{k+1}=\nabla q(x_{k+1})=Hx_{k+1}+c$.} In each iteration, the
dimension of the affine subspace where the optimal solution is sought 
increases by one. After at most $n$ iterations the optimal solution in
$\mathbb{R}^n$ is found, which will
then be the optimal solution of (QP).\footnote{It may happen, depending on the
number of distinct eigenvalues of $H$ and the orientation of $p_0$,
that the optimal solution is found after less than $n$ iterations, see, e.g.,
\cite[Chapter 6]{saad}.} 

The search direction $p_k$ belongs to $\mathcal{K}_{k+1}(p_0,H)$, and as it is conjugate to all
previous search directions with respect to $H$ it holds that $span\{p_0, p_1, \dots, p_k\}=\mathcal{K}_{k+1}(p_0, H)$,
i.e., the search directions $p_0, \dots, p_k$ form an $H$-orthogonal
basis for $\mathcal{K}_{k+1}(p_0,H)$. We will henceforth refer to the search direction produced by CG, in
iteration $k$ on a given \eqref{qp}, as $p_k^{CG}$.

Since the gradients are mutually orthogonal, and because of the relationship
with the search directions in \eqref{cgdir2}, it can be shown that $span\{g_0,
\dots, g_k\}=\mathcal{K}_{k+1}(p_0,H)$, i.e. the gradients form an orthogonal
basis for $\mathcal{K}_{k+1}(p_0,H)$.

General conjugate direction methods can not be described as
Krylov subspace methods, since in general $span\{p_0, \dots,
p_{k}\} \neq \nolinebreak\mathcal{K}_{k+1}(p_0,H)$. We will use this
special characteristic of CG when
investigating the connection to QN.

Although our focus is quadratic programming it deserves mentioning that CG was extended to general unconstrained problems by
Fletcher and Reeves \cite{fletcherreeves}.

\subsection{Quasi-Newton methods}\label{sec:qn}

In QN methods the search directions are generated by solving
\begin{equation}\label{calp}
B_kp_k=-g_k,
\end{equation}
in each iteration $k$, where the matrix $B_k$ is chosen to be an approximation of $H$ in
some sense.\footnote{The
choice $B_k=H$ would give Newton's method, whereas the choice $B_k=I$
would give the steepest-descent method.} In this paper, we will consider \emph{symmetric} approximations of the Hessian, 
i.e. $B_k=B_k^T$. It is also possible to consider unsymmetric
approximation matrices, see, e.g., \cite{huang}. 

The first suggestion of a QN method was made by Davidon in 1959
\cite{davidon}, using the term \emph{variable metric method}. In
1963, in a famous paper by Fletcher and Powell \cite{fletcherpowell}, Davidon's
method was modified\footnote{"We have made both a simplification by
  which certain orthogonality conditions which are important to the
  rate of attaining the solution are preserved, and also an
  improvement in the criterion of convergence.'' \cite{fletcherpowell}} and this was the starting point for making
these QN methods widely known, used and studied. 

We choose to work with an approximation of the Hessian rather than
an approximation of the inverse Hessian, $M_k$, as many of the earlier papers
did, e.g. \cite{fletcherpowell}. Our results can however straightforwardly be derived
for the inverse point of view where \eqref{calp} is replaced by the
equation $p_k=-M_kg_k$.

The approximation matrix $B_k$ used in iteration $k$ to solve for $p_k$ is obtained by adding an
\emph{update matrix}, $U_k$, to the previous approximation matrix,
\begin{equation}\label{Bupdate}
B_k=B_{k-1}+U_k.
\end{equation}

One often considers the Cholesky factorization of $B_k$, then
\eqref{calp} can be solved in order of $n^2$ operations. Also, if in
\eqref{Bupdate} the update matrix $U_k$ is of low-rank, one does not need to compute the Cholesky factorization
of $B_k$ from scratch in each iteration, see, e.g., \cite{practicalopt}.

One of the most well-known update schemes is the one using update
matrices from the \emph{one-parameter Broyden
family of updates} \cite{broyden1967} described by
\begin{equation}\label{broyden}
U_k= \frac{Hp_{k-1}p_{k-1}^TH}{p_{k-1}^T Hp_{k-1}}-\frac{B_{k-1}p_{k-1}p_{k-1}^TB_{k-1}}{p_{k-1}^T B_{k-1}p_{k-1}}+\phi_k p_{k-1}^TB_{k-1}p_{k-1} ww^T,
\end{equation}
with 
$$
w=\frac{Hp_{k-1}}{p_{k-1}^T Hp_{k-1}}-\frac{B_{k-1}p_{k-1}}{p_{k-1}^T B_{k-1}p_{k-1}},
$$
and where $\phi_k$ is a free parameter, known as the
\emph{Broyden parameter}. 
Equation \eqref{broyden} may be written more compactly, momentarily dropping all
subscripts, as
\begin{equation}\label{broyden2}
U=\left(
\begin{array}{ll} 
\frac{1}{p^T Hp}Hp & \frac{1}{p^T Bp}Bp
\end{array} \right) 
\left(
\begin{array}{cc} 
p^THp+\varphi & -\varphi \\
-\varphi  & -p^TBp+\varphi
\end{array} \right) 
\left(
\begin{array}{c} 
\frac{1}{p^T Hp}p^TH\\ 
\frac{1}{p^T Bp}p^TB 
\end{array} \right),
\end{equation}
where $\varphi=\phi_k p^TBp$. It is common to express the
one-parameter Broyden family in terms of $y_k=g_{k+1}-g_{k}$
and $s_k=x_{k+1}-x_{k}$, see, e.g. \cite{fletcherpractical}, but as our interest is in the search directions
we prefer the equivalent form of \eqref{broyden}.

For all updates in this family, \eqref{Bupdate} has the property of
\emph{hereditary symmetry}, i.e. if $B_{k-1}$ is symmetric then $B_k$
will be symmetric. The update given by the choice $\phi_k=0$ is known as
the \emph{Broyden-Fletcher-Goldfarb-Shanno}-update, or BFGS-update 
for short. For this update, when exact linesearch is used, \eqref{Bupdate} has the property of \emph{hereditary
  positive definiteness}, i.e. if $B_{k-1} \succ 0$ then $B_{k} \succ
0$. An implication of this is that for all
updates given by $\phi_k \geq 0$, when exact linesearch is used, \eqref{Bupdate} has the property of hereditary
  positive definiteness, see,
e.g., \cite[Chapter 9]{luenberger}. Note that there are updates in the
one-parameter Broyden family for which \eqref{Bupdate} does not have
this property.

\subsection{Background results}\label{sec:background results}
In \cite{dixon}, Dixon has shown that on any smooth function,
using perfect linesearch\footnote{A generalization of exact
  linesearch for general smooth functions, see \cite{dixon}.}, the one-parameter Broyden family gives rise
to parallel search directions. On \eqref{qp}, these
search directions will in addition be mutually conjugate with respect to
$H$, see, e.g., \cite{huang}.

On \eqref{qp} it is well-known that these conjugate search directions,
generated by the one-parameter Broyden family, will be parallel to
those of CG, i.e. $p_k \myparallel p_k^{CG}$, for all $k$. See, e.g.,
\cite{fletcherpractical, huang, nazareth}. In particular Theorem 3.4.2
on page 65 of \cite{fletcherpractical} states that CG and QN using
well-defined update matrices from the one-parameter Broyden family
generate identical iterates. Note that this connection between CG and
QN does not hold for general convex functions.

In this paper we approach the well-known connection between QN
and CG from another perspective. The main question handled in this
paper is: when solving \eqref{qp}, what are the precise conditions on
$B_k$ and $U_k$ such that $p_k \myparallel p_k^{CG}$ is obtained? We
provide an answer by turning our attention to the search directions
defined by CG and the Krylov subspaces they span. In
Proposition~\ref{U_cg_2} we state explicit conditions on $U_k$ and in
Theorem~\ref{thmbroyden} we extend the well-known connection between
CG and QN as we show that, under a sufficient condition to generate
conjugate search directions, no other update matrices than those in
the one-parameter Broyden family will make $p_k \myparallel p_k^{CG}$.

\section{Results}\label{sec:results}
As a reminder for the reader we
will, in the following proposition, state the necessary and sufficient
conditions for a vector $p_k$ to be parallel to the vector
$p_k^{CG}$ given that $p_i \myparallel
p_i^{CG}$, for all $i \leq k-1$. Given $x_0$, one may calculate $g_0=Hx_0+c$.
\begin{proposition}\label{p_cg}
Let $p_0=p_0^{CG}=-g_0$ and $p_i \myparallel
p_i^{CG}$, for all $i \leq k-1$. Assume $p_k \neq 0$, then $p_k \myparallel p_k^{CG}$ if and only if
\begin{itemize}
\item[(i)] $p_k \in \mathcal{K}_{k+1}(p_0,H)$, and,
\item[(ii)] $p_k^THp_i=0, \quad  \forall i \leq k-1$.
\end{itemize}
\end{proposition}
\begin{proof} 
For the sake of completeness we include the proof.

 \emph{Necessary:} Suppose $p_k$ is parallel to $p_k^{CG}$.
 Then $p_k=\delta_k p_k^{CG}$, for an arbitrary nonzero
  scalar $\delta_k$.  As $p_k^{CG}$
  satisfies (i), it holds that 
$$
p_k=\delta_k p_k^{CG} \in  \mathcal{K}_{k+1}(p_0,H),
$$
since $ \mathcal{K}_{k+1}(p_0,H)$ is a
  linear subspace.
And as $p_k^{CG}$ satisfies (ii), it follows that,
$$
p_k^THp_i=\delta_k \delta_i (p_k^{CG})^TH p_i^{CG}=0, \quad \forall i
\leq k-1.
$$

\emph{Sufficient:} Suppose $p_k$ satisfies (i) and (ii). The set of
vectors $\{p_0^{CG}, \dots, p_{k}^{CG}\}$ form an $H$-orthogonal basis for the
space $\mathcal{K}_{k+1}(p_0,H)$ and the set of vectors $\{p_0, \dots, p_{k-1}\}$ form an $H$-orthogonal basis for the
space $\mathcal{K}_{k}(p_0,H)$. Since $p_k$ satisfies (i) and (ii) it
must hold that $p_k \myparallel p_k^{CG}$.  
\end{proof}

These necessary and sufficient conditions will serve as a foundation for the rest of
our results. We will determine conditions, first on $B_k$, and
second on $U_k$ used in iteration $k$ of QN, in order for $p_k \myparallel p_k^{CG}$. 

Given the current iterate $x_k$, one may calculate
$g_k=Hx_k+c$. Hence, in iteration $k$ of QN, $p_k$ is determined from
\eqref{calp} and depends entirely 
on the choice of $B_k$. We make the assumption that \eqref{calp} is
compatible, i.e. that a solution exists. A well-known sufficient condition for $p_k$ to satisfy (ii) of
Proposition \ref{p_cg} is for $B_k$ to satisfy
\begin{equation}\label{conj}
B_kp_i=H p_i, \quad \forall i \leq k-1.
\end{equation}
This condition, known as the \emph{hereditary condition}, will be used in the following results, which implies that the stated conditions on $B_k$ and later
$U_k$ will be \emph{sufficient} conditions, and not necessary and sufficient
as those in Proposition \ref{p_cg}. All $B_k$ satisfying \eqref{conj} can be seen as defining
different conjugate direction methods, but only certain choices of
$B_k$ will generate conjugate directions parallel to those of CG. 

Assume that $B_k$ is constructed as,
\begin{equation}\label{Bconst}
B_k=I +V_k,
\end{equation}
i.e. an identity matrix plus a matrix $V_k$.\footnote{Any symmetric
  matrix can be expressed like this.} We make no assumptions on $V_k$ except symmetry. 

Let $B_0=I$, then $p_0=p_0^{CG}=-g_0$. Given $p_i \myparallel
p_i^{CG}$, for all $i \leq k-1$, the
following lemma gives sufficient
conditions on $B_k$ in order for $p_k \myparallel p_k^{CG}$.

\begin{lemma}\label{B_cg}
Let $B_0=I$, so that $p_0=p_0^{CG}=-g_0$.  Assume that $p_i \myparallel
p_i^{CG}$, for all $i \leq k-1$, and that a solution $p_k$ is obtained from
\eqref{calp}. Then $p_k \myparallel p_k^{CG}$ if $B_k$ satisfies
\begin{itemize}
\item[(i)] $\mathcal{R}(V_k) \subseteq  \mathcal{K}_{k+1}(p_0,H)$, and,
\item[(ii)] $B_kp_i=Hp_i, \quad \forall i \leq k-1$,
\end{itemize}
where $\mathcal{R}(V_k)=\{y\ : \ y=V_kx\}$, the \emph{range-space} of
the matrix $V_k$.

\end{lemma} 
\begin{proof}
First we show that $p_k$ satisfies condition (i) of Proposition \ref{p_cg}. If
\eqref{Bconst} is inserted into \eqref{calp} one obtains,
$$
\big(I +V_k\big)p_k=-g_k,
$$
so that
$$
p_k=-g_k-V_kp_k.
$$
Since $g_k \in \mathcal{K}_{k+1}(p_0,H)$,
and $\mathcal{R}(V_k) \subseteq \mathcal{K}_{k+1}(p_0,H)$, it holds that $p_k \in \mathcal{K}_{k+1}(p_0,H)$. Hence, $p_k$
satisfies condition (i) of Proposition \ref{p_cg}. Secondly, $p_k$ satisfies condition (ii) of Proposition \ref{p_cg}
since condition (ii) of this lemma is identical to \eqref{conj}.

Hence, $p_k \myparallel p_k^{CG}$ by Proposition \ref{p_cg}.  
\end{proof}

As $B_k$ is updated according to \eqref{Bupdate}, one would prefer to
have conditions, in iteration $k$, on the update matrix $U_k$ instead of on the
entire matrix $B_k$. Therefore, we now modify Lemma \ref{B_cg} by
noting that equation \eqref{Bconst} may be stated as
\begin{equation}\label{Uconst}
B_k=B_{k-1}+U_k=I+V_{k-1} +U_k.
\end{equation}
We make no
assumptions on $U_k$ except symmetry.
Note that one
may split \eqref{conj}  as 
\begin{equation}\label{conjrest}
B_kp_{i}=Hp_{i},  \quad \forall i \leq k-2,
\end{equation}
\begin{equation}\label{conjqn}
B_kp_{k-1}=Hp_{k-1}.
\end{equation}
Equation \eqref{conjqn} is known as the \emph{quasi-Newton
  condition}.\footnote{In the literature, much emphasis is placed on that updates
should satisfy \eqref{conjqn}. This condition alone is
  not a sufficient condition on $B_k$ to give conjugate
  directions. } Using \eqref{Bupdate} one can reformulate \eqref{conjrest} and \eqref{conjqn} in terms of
$U_k$, see (ii) and (iii) of the following proposition. 

Given $B_{k-1}$ that satisfies condition (i) of
Lemma \ref{B_cg}, and given $p_i \myparallel
p_i^{CG}$, for all $i \leq k-1$, the
following proposition gives sufficient
conditions on $U_k$ in order for $p_k \myparallel p_k^{CG}$.

\begin{proposition}\label{U_cg_2}
Let $B_0=I$, so that $p_0=p_0^{CG}=-g_0$. Assume that $p_i \myparallel
p_i^{CG}$, for all $i \leq k-1$, that $B_{k-1}$ satisfies condition
(i) of Lemma \ref{B_cg}, that a solution $p_k$ is obtained from
\eqref{calp} and that $B_k$ is obtained from \eqref{Uconst}. Then
$p_k \myparallel p_k^{CG}$  if $U_k$ satisfies
\begin{itemize}
\item[(i)] $\mathcal{R}(U_k) \subseteq  \mathcal{K}_{k+1}(p_0,H)$,
\item[(ii)] $U_k p_i=0, \quad  \forall i \leq k-2$, and,
\item[(iii)] $U_kp_{k-1}=(H-B_{k-1})p_{k-1}$.
\end{itemize}
\end{proposition} 

\begin{proof}
The proof is identical to the one of Lemma \ref{B_cg}. If
\eqref{Uconst} is inserted into \eqref{calp} one obtains
$$
p_k=-g_k-V_{k-1}p_k-U_kp_k.
$$

By condition (i) of Lemma \ref{B_cg}, $\mathcal{R}(V_{k-1}) \subseteq \mathcal{K}_{k}(p_0,H)$, and since \linebreak$\mathcal{K}_{k}(p_0,H) \subseteq
\mathcal{K}_{k+1}(p_0,H)$, it follows that $V_{k-1}p_k \in \mathcal{K}_{k+1}(p_0,H)$.
Therefore, since $g_k \in \mathcal{K}_{k+1}(p_0,H)$ and $\mathcal{R}(U_k) \subseteq 
\mathcal{K}_{k+1}(p_0,H)$, it holds that $p_k \in
\mathcal{K}_{k+1}(p_0,H)$. Hence, $p_k$ satisfies condition (i) of Proposition \ref{p_cg}.

Conditions (ii) and (iii) of this lemma are merely a reformulation of \eqref{conj}, hence $p_k$ satisfies condition (ii) of Proposition \ref{p_cg}. 

Hence, $p_k \myparallel p_k^{CG}$ by Proposition \ref{p_cg}.  
\end{proof}

The assumption of the previous proposition, that $B_{k-1}$ is chosen to satisfy condition (i) of Lemma \ref{B_cg},
will be satisfied if, in each iteration $k$, the update matrix $U_k$ is chosen
according to conditions (i)-(iii) of Proposition \ref{U_cg_2}. This is summarized in the
following corollary.

\begin{corollary}
If in each iteration $k$ the update matrix $U_k$ is chosen to satisfy
conditions (i)-(iii) of Proposition \ref{U_cg_2}, then $p_k \myparallel
p_k^{CG}$ for all $k$.
\end{corollary}

Next we state a result, which will be needed in our further
investigation, and holds for any update scheme of $B_k$, which generates search
directions parallel to those of CG. 

\begin{proposition}\label{Bdefined}
If $p_k \myparallel
p_k^{CG}$ for all $k$ and a solution $p_k$ obtained from
\eqref{calp}, then $p_{k}^TB_{k}p_{k} \neq
0$, unless $B_{k}p_{k}=-g_{k}=\nolinebreak 0$.
\end{proposition}
\begin{proof}
Since $p_k=\delta_k p_k^{CG}$, for some non-zero scalar $\delta_k$,
and since $p_k^{CG}$ is a descent direction it holds that
$$
p_{k}^TB_{k}p_{k}=-p^T_{k}g_{k}=-\delta_k (p_k^{CG})^Tg_k \neq 0,
$$
for some non-zero scalar $\delta_k$, unless $g_k=0$.  
\end{proof}

Note that this implies that \eqref{nonorth} is
satisfied for any QN method using an update scheme that generates
search directions that are parallel to those of CG. Also, assuming $g_k
\neq 0$, it implies that the fraction $\frac{1}{p_{k}^TB_{k}p_{k}}$ is well-defined.

\subsection{Update matrices defined by Proposition
  \ref{U_cg_2}}\label{sec:U matrix}

Having stated precise conditions on $U_k$ in Proposition \ref{U_cg_2}
we now turn to look at what these conditions imply
in terms of actual update matrices. In Theorem \ref{thmbroyden}, we show
that the conditions on $U_k$ in Proposition \ref{U_cg_2} are
equivalent to the matrix $U_k$ belonging to the one-parameter
Broyden family, \eqref{broyden2}. This implies that, under the
sufficient condition \eqref{conj}, there are no other update matrices,
even of higher rank, that make $p_k \myparallel p_k^{CG}$.

The if-direction of Theorem \ref{thmbroyden}, the fact that the one-parameter
Broyden family satisfies the conditions of Proposition \ref{U_cg_2} is
straightforward. However, the only if-directions
shows that there are no update matrices outside this family that satisfy the conditions of Proposition \ref{U_cg_2}.

\begin{theorem}\label{thmbroyden}
Assume $g_k \neq 0$. A matrix $U_k$ satisfies (i)-(iii) of Proposition \ref{U_cg_2}  if and only if $U_k$
can be expressed according to \eqref{broyden2}, the one-parameter
Broyden family, for some $\phi_k$.
\end{theorem}

\begin{proof}
Note that since $\mathcal{K}_{k-1}(p_0,H)=span\{p_0,
\dots, p_{k-2}\}$, condition (ii) of
Proposition \ref{U_cg_2} can be stated as
$$
\mathcal{K}_{k-1}(p_0,H) \subseteq \mathcal{N}(U_k),
$$
where $\mathcal{N}(U_k)=\{x\ : \ U_kx=0\}$, the \emph{null-space} of $U_k$. This implies that $\dim(\mathcal{N}(U_k)) \geq k-1$. Since $U_k$
is symmetric and applying condition (i) of
Proposition \ref{U_cg_2} it follows that
$$
\mathcal{R}(U_k) \subseteq \mathcal{K}_{k-1}(p_0,H)^{\perp} \cap
\mathcal{K}_{k+1}(p_0,H)=span\{g_{k-1}, g_k\},
$$
and that $\dim(\mathcal{R}(U_k^T))=\dim(\mathcal{R}(U_k)) \leq 2$.

Hence, one may write a
general $U_k$ that satisfies the conditions (i) and (ii) of Proposition \ref{U_cg_2} as
\begin{equation}\label{broyden3}
U_k=\left(
\begin{array}{ll} 
g_{k-1} & g_k
\end{array} \right) 
\left(
\begin{array}{cc} 
m_{1,1} & m_{1,2}\\
m_{1,2} & m_{2,2}
\end{array} \right)
\left(
\begin{array}{c} 
g_{k-1}^T\\ 
g_k^T 
\end{array} \right).
\end{equation}

Note that due to the linear relationship between $\{
B_{k-1}p_{k-1}, Hp_{k-1}\}$  and $\{g_{k-1}, g_k\}$ it holds that
\begin{equation*}
g_{k-1}=-B_{k-1}p_{k-1}, \quad g_k=\alpha_{k-1}Hp_{k-1}+g_{k-1}.
\end{equation*}
Hence, since $\alpha_{k-1}=(p_{k-1}^T B_{k-1} p_{k-1})/(p_{k-1}^T H
p_{k-1})$ by the exact linesearch, dropping all $'k-1'$-subscripts and
letting $g_k$ be represented by $g_+$, we obtain
\begin{equation}\label{rel1}
\left(
\begin{array}{c} 
g^T \\
g_+^T
\end{array} \right)
= p^TBp\left(
\begin{array}{cc} 
0 & \ -1\\
1 & \ -1
\end{array} \right)
\left(
\begin{array}{c} 
\frac{1}{p^T Hp}p^TH\\ 
\frac{1}{p^T Bp}p^TB 
\end{array} \right).
\end{equation}
We may therefore rewrite \eqref{broyden3} as
\begin{equation}\label{broyden4}
U_k=\left(
\begin{array}{ll} 
\frac{1}{p^T Hp}Hp & \frac{1}{p^T Bp}Bp
\end{array} \right)  
\left(
\begin{array}{cc} 
\hat{m}_{1,1} & \hat{m}_{1,2}\\
\hat{m}_{1,2} & \hat{m}_{2,2}
\end{array} \right)
\left(
\begin{array}{c} 
\frac{1}{p^T Hp}p^TH\\ 
\frac{1}{p^T Bp}p^TB 
\end{array} \right).
\end{equation}

Now imposing condition (iii) of Proposition \ref{U_cg_2}, 
\eqref{broyden4} yields
\begin{equation}\label{relUp}
U_kp=\left(
\begin{array}{ll} 
\frac{1}{p^T Hp}Hp & \frac{1}{p^T Bp}Bp
\end{array} \right)  
\left(
\begin{array}{cc} 
\hat{m}_{1,1} & \hat{m}_{1,2}\\
\hat{m}_{1,2} & \hat{m}_{2,2}
\end{array} \right)
\left(
\begin{array}{c} 
1\\ 
1 
\end{array} \right)=Hp-Bp.
\end{equation}
A combination of Proposition~\ref{Bdefined} and \eqref{rel1} shows that
$Hp$ and $Bp$ are linearly independent. Hence, \eqref{relUp} implies
that
$$
\hat{m}_{1,1}+\hat{m}_{1,2}=p^THp , \quad \hat{m}_{1,2}+\hat{m}_{2,2}=-p^TBp.
$$
With $\hat{m}_{1,2}=-\varphi$, then 
$$
\hat{m}_{1,1}=p^THp+\varphi , \quad \hat{m}_{2,2}=-p^TBp+\varphi.
$$
Substituting into \eqref{broyden4}, we obtain \eqref{broyden2}
with the scaling $\varphi=\phi_k p^TBp$. This completes the proof.  
\end{proof}

The precise conditions of Proposition \ref{U_cg_2} are equivalent to choosing
$U_k$ from the one-parameter Broyden family. As seen in the proof, the
conditions (i)-(iii) of Proposition \ref{U_cg_2} may be expressed as 
\begin{itemize}
\item[(a)]$\mathcal{R}(U_k) \subseteq  span\{g_{k-1}, g_k\}=\mathcal{K}_{k-1}(p_0,H)^{\perp} \cap
\mathcal{K}_{k+1}(p_0,H)$, and,
\item[(b)]$U_kp_{k-1}=(H-B_{k-1})p_{k-1}$.
\end{itemize}

Note that after we impose (a) we have three degrees of
freedom, $\hat{m}_{1,1}$, $\hat{m}_{1,2}$ and $\hat{m}_{2,2}$. Then after
imposing (b), the \emph{quasi-Newton condition}, we are left with one
degree of freedom $\phi_k$.

What is often mentioned as a feature of
the one-parameter Broyden family, using information only from the
current and previous iteration when forming $U_k$, is
in fact a condition that guarantees the equivalence to CG.  Also, the fact
that the update matrices in \eqref{broyden2} are of rank at most two is a
consequence of satisfying this condition. This is what distinguishes the one-parameter
Broyden family of updates from any $B_k$ satisfying
\eqref{conj}. 

We stress that the fact that CG and QN, using a well-defined update
matrix from the one-parameter Broyden family, generates parallel
search directions and hence identical iterates is well-known. By
Proposition~\ref{U_cg_2} and Theorem~\ref{thmbroyden} we may draw the additional
conclusion that, under the
sufficient condition \eqref{conj}, there are no other update matrices,
even of higher rank, that make $p_k \myparallel p_k^{CG}$.

\subsection{Relation between $\delta_k$ and
  $\phi_k$}\label{sec:relation}
Next we derive a relation between the free parameter $\phi_k$, the Broyden
parameter used in $U_k$ when forming $B_k$ according to \eqref{Bupdate}, and the non-zero
parameter $\delta_k$ in $p_k=\delta_k p_k^{CG}$. In the following lemma we state an alternative way to express $U_k$
belonging to the one-parameter Broyden family \eqref{broyden2}. 

\begin{lemma}\label{U_g}
An update matrix $U_k$ from the one-parameter Broyden family
\eqref{broyden2} can be expressed on the form
\begin{equation}\label{broyden5}
U_k=\left(
\begin{array}{ll} 
g_{k-1} & g_k
\end{array} \right) 
\left(
\begin{array}{cc} 
\frac{p^THp}{(p^TBp)^2} -\frac{1}{p^TBp}& -\frac{p^THp}{(p^TBp)^2}\\
-\frac{p^THp}{(p^TBp)^2}& \frac{p^THp}{(p^TBp)^2}+\frac{\phi_k}{p^TBp}
\end{array} \right)
\left(
\begin{array}{c} 
g_{k-1}^T\\ 
g_k^T 
\end{array} \right),
\end{equation}
where all omitted subscripts
are $'k-1'$.
\end{lemma}
\begin{proof}
Reversing relation \eqref{rel1} gives
\begin{equation}\label{rel2}
\left(
\begin{array}{c} 
\frac{1}{p^T Hp}p^TH \\
\frac{1}{p^T Bp}p^TB 
\end{array} \right)
= \frac{1}{p^TBp}\left(
\begin{array}{cc} 
-1 & \ 1\\
-1 & \ 0
\end{array} \right)
\left(
\begin{array}{c} 
g_{k-1}^T\\ 
g_k^T
\end{array} \right).
\end{equation}
If \eqref{rel2} is inserted in \eqref{broyden2} with
$\varphi=\phi_k p^TBp$, then \eqref{broyden5} is obtained.
\end{proof}

In the following
proposition we show the one-to-one correspondence between the two
parameters $\phi_k$ and $\delta_k$, 

\begin{proposition}\label{phi delta}
Let $B_0=I$, so that $p_0=p_0^{CG}=-g_0$. Assume that $p_i \myparallel
p_i^{CG}$, for all $i \leq k-1$, that $B_{k-1}$ satisfies condition
(i) of Lemma \ref{B_cg}, that a solution $p_k$ is obtained from
\eqref{calp}, that $B_k$ is obtained from \eqref{Uconst} and that
$U_k$ belongs to the one-parameter Broyden family, \eqref{broyden2},
for some $\phi_k$, where $\phi_k \neq -(p_{k-1}^TB_{k-1}p_{k-1})/(g_k^Tg_k)$. 

Then the non-zero parameter $\delta_k$ in $p_k=\delta_k p_k^{CG}$ is
given by
\begin{equation}\label{deltabroyden}
\delta_k(\phi_k)=\frac{1}{1+\phi_k\frac{g_k^Tg_k}{p_{k-1}^TB_{k-1}p_{k-1}}}.
\end{equation}
\end{proposition}

\begin{proof}
Omitting all $'k-1'$-subscripts throughout the proof. Our assumption is that $p_k$ is given by
$$
p_k=-g_k-V_{k-1}p_k -U_kp_k.
$$
By Theorem \ref{thmbroyden} $p_k \myparallel
p_k^{CG}$ hence,
$$
\delta_k p_k^{CG}=-g_k-V_{k-1}\delta_k p_k^{CG} -U_k \delta_k p_k^{CG}.
$$
If the above expression is projected onto $g_k$ it holds that
$$
g_k^T(\delta_k p_k^{CG})=-g_k^Tg_k-\delta_k g_k^TV_{k-1} p_k^{CG} - \delta_k g_k^TU_k p_k^{CG}=
-g_k^Tg_k-0- \delta_k g_k^TU_k p_k^{CG},
$$
since by condition (i) of Lemma \ref{B_cg}, $\mathcal{R}(V_{k-1})
\subseteq \mathcal{K}_{k}(p_0,H)=span\{g_0, g_1, \dots, g_{k-1}\}$.
Using \eqref{broyden5} for $U_k$ it follows that
\begin{equation*}
\delta_k g_k^T p_k^{CG}=-g_k^Tg_k- \delta_k g_k^T g_k
\left(
\begin{array}{cc} 
-\frac{p^THp}{(p^TBp)^2} & \quad \frac{p^THp}{(p^TBp)^2}+\frac{\phi_k}{p^TBp}
\end{array} \right)
\left(
\begin{array}{c} 
g_{k-1}^T p_k^{CG}\\ 
g_k^T p_k^{CG} 
\end{array} \right).
\end{equation*}

From \eqref{cgdir2} it follows that $g_i^Tp_k^{CG}=-g_k^Tg_k$, for all $i \leq k$, which implies that
\begin{equation*}
-\delta_k g_k^Tg_k=-g_k^Tg_k- \delta_k g_k^T g_k
\left(
\begin{array}{cc} 
-\frac{p^THp}{(p^TBp)^2} & \quad \frac{p^THp}{(p^TBp)^2}+\frac{\phi_k}{p^TBp}
\end{array} \right)
\left(
\begin{array}{c} 
-g_k^Tg_k\\ 
-g_k^Tg_k 
\end{array} \right), 
\end{equation*}
dividing on both sides with $-g_k^Tg_k$ yields
\begin{equation*}
\delta_k =1- \delta_k g_k^T g_k
\big(
-\frac{p^THp}{(p^TBp)^2} + \frac{p^THp}{(p^TBp)^2}+\frac{\phi_k}{p^TBp}
\big)=1- \delta_k g_k^T g_k\frac{\phi_k}{p^TBp},
\end{equation*}
which implies \eqref{deltabroyden}.
 
\end{proof}
Note that this result implies that for $\phi_k=0$ we get $\delta_k=1$,
i.e. $p_k^{BFGS}=p_k^{CG}$ as long as $\phi_i$, for all $i \leq k-1$,
are well-defined\footnote{See the next paragraph and Section \ref{sec:broyden defined}.}. In \cite{nazareth}, Nazareth
derives the relation between $p_k^{BFGS}$ and $p_k^{CG}$ using
induction. 

Note that substituting the limit $\phi_k \to -(p_{k-1}^TB_{k-1}p_{k-1})/(g_k^Tg_k)$ into
\eqref{deltabroyden} for $g_k\ne \nolinebreak 0$ yields $\delta_k \to \infty$.
This limit of $\phi_k$ is called a \emph{degenerate value}
since it makes $B_k$
singular, which is a contradiction to our assumption that
$B_kp_k=-g_k$ is compatible. For all other values of $\phi_k$, we get a one-to-one
correspondence between $\phi_k$ and $\delta_k$ and this implies that our assumption that
$B_kp_k=-g_k$ is compatible actually implies that the solution $p_k$
is unique.

The assumption of Proposition \ref{phi delta}, that $B_{k-1}$ is chosen to satisfy condition (i) of Lemma \ref{B_cg},
will be satisfied if, in each iteration $k$, the update matrix $U_k$
belongs to the one-parameter Broyden family, \eqref{broyden2}, for
some $\phi_k$.

\subsection{Remarks on when the one-parameter
  Broyden family is well-defined}\label{sec:broyden defined}

There are several ways in which an update scheme using the
one-parameter Broyden family may become not well-defined. We have
already mentioned the degenerate value that makes $B_k$ singular in
the previous section. 

In addition to the above, the matrix $U_k$ may itself become not well-defined.
For all well-defined values of $\phi_k$ it holds $U_k$ given by \eqref{broyden2}
is not well-defined if and only if $p_{k-1}^TB_{k-1}p_{k-1}=0$ and $B_{k-1}p_{k-1} \neq 0$. 
It is clear that requiring $B_{k-1}$ to be definite (positive or negative) is sufficient
to avoid $U_k$ being not well-defined. However, from Proposition
\ref{Bdefined} it follows that, on (QP), $p_{k-1}^TB_{k-1}p_{k-1}=0$ and $B_{k-1}p_{k-1} \neq 0$
does not occur for any update scheme which generates search
directions parallel to those of CG.\footnote{For general functions we
  may have $p_{k-1}^TB_{k-1}p_{k-1}=0$ and $B_{k-1}p_{k-1} \neq 0$, and in \cite{fletchersinclair} the
  values of $\phi_k$ that give rise to this situation are characterized
and also termed degenerate values.}

Hence, on \eqref{qp}, in order for $U_k$, given by \eqref{broyden2}, to become not
well-defined, the undefinedness must enter in the Broyden parameter
$\phi_k$. It is well-known that the symmetric rank-one update, SR1, may
become not well-defined on (QP), see, e.g. \cite[Chapter
9]{luenberger}. SR1 is uniquely determined by
\eqref{conjqn}\footnote{An illustration that the conditions (i)-(iii)
  of Proposition \ref{U_cg_2} are indeed
  only sufficient.} and the Broyden parameter for SR1 is given by $\phi^{SR1}_k=(p_{k-1}^THp_{k-1})/(p_{k-1}^T(H-B_{k-1})p_{k-1})$,
an expression that becomes undefined for $p_{k-1}^T(H-B_{k-1})p_{k-1}=0$. Note that
this is equivalent to $\alpha_{k-1}=(p_{k-1}^TB_{k-1}p_{k-1})/(p_{k-1}^THp_{k-1})=1$. 

Hence, we may
summarize our remarks on when $U_k$ given by the one-parameter Broyden
family, \eqref{broyden2}, is well-defined in the following corollary.

\begin{corollary}
Unless $\phi_k=\phi^{SR1}_k$ in \eqref{broyden2} with the unit
steplength taken in the same iteration, then $U_k$ defined by \eqref{broyden2}, the one-parameter
Broyden family, is always well-defined on \eqref{qp}.
\end{corollary}

Hence, taking the unit steplength is an indication that one needs to choose a
different update scheme than SR1 when forming $B_{k}$. We therefore
stress that hereditary (positive or negative) definiteness is not a
necessary property for the update matrices $U_k$ to be well-defined
when solving \eqref{qp}. 

\section{Conclusions}\label{sec:conclusion}

The main result of this paper are the precise conditions on the update
matrix $U_k$ stated in Proposition \ref{U_cg_2}. In addition, from
Theorem \ref{thmbroyden}, we draw the conclusion that, in the
framework where we use the sufficient condition \eqref{conj} to
guarantee conjugacy of the search directions, the update schemes in QN
that give parallel search directions to those of CG are completely
described by the one-parameter Broyden family. Hence, we are able to
extend the well-known connection between CG and QN, e.g., given in
\cite{fletcherpractical}. We show that, under the sufficient condition
\eqref{conj}, there are no other update matrices, even of higher rank,
that make $p_k \myparallel p_k^{CG}$.

It seems as it may be the sufficient requirement to
get conjugate directions, \eqref{conj} that limits the freedom when choosing $U_k$ to the one-parameter Broyden
family of updates as shown by the above results. Since
this condition is only sufficient, one may still pose the
question if there are other update schemes for QN that yield the same sequence of
iterates as CG on \eqref{qp}. We believe that in order to understand
this possible limitation it will be necessary to obtain a deeper
understanding of CG.

We derive a one-to-one correspondence between the Broyden parameter
$\phi_k$ and the scaling $\delta_k$. This relation implies that the
assumption that \eqref{calp} is compatible implies uniqueness of
the solution $p_k$. If the degenerate value is used $B_k$ becomes singular.

We are also able to make the remark that the update matrices belonging
to the one-parameter Broyden
family is always well-defined on \eqref{qp}, unless the
steplength is of unit length and in the same iteration the rank-one
update is used. In this case it is the Broyden parameter $\phi_k$ that
becomes undefined.

In this paper we have focused on quadratic programming. Besides being
important in its own right, it is also a highly important as a subproblem when
solving unconstrained optimization problems. For a survey on methods
for unconstrained optimization see, e.g., \cite{recentadv}. Also, it
deserves mentioning that work has been done on QN
update schemes for general unconstrained optimization
considering \eqref{conj} for only $j=k-1$ and $j=k-2$, deriving an
update scheme that satisfies the quasi-Newton condition and has a
minimum violation of it for the previous step, see
\cite{mifflinnazareth, nazareth2}. 

A further motivation for this research in this paper is that the deeper understanding
of what is important in the choice of $U_k$ could be implemented in a
limited\nolinebreak-memory QN method. I.e., can one choose which columns to save
based on some other criteria than just picking the most recent ones? See, e.g., \cite[Chapter 9]{nocedalwright}, for an introduction to limited\nolinebreak-memory QN methods.

Finally, it should be pointed out that the discussion of this paper is
limited to exact arithmetic. Even in cases where CG and QN generate
identical iterates in exact arithmetic, the difference between
numerically computed iterates by the two methods may be quite large,
see, e.g.~\cite{hager}.

{\small
\subsubsection*{Acknowledgements:}
We thank the referees and the associate editor for their insightful
comments and suggestions which significantly improved the presentation
of this paper. 
}

\bibliography{references}{}
\bibliographystyle{acm}

\end{document}